\documentclass[a4paper]{amsart}
\usepackage[T1]{fontenc}
\usepackage[british]{babel}
\usepackage{amssymb}
\usepackage{mathtools}
\usepackage{bm}
\usepackage[hidelinks]{hyperref}
\let\oldqedsymbol\qedsymbol
\DeclareMathOperator{\Con}{Con}
\DeclareMathOperator{\tok}{tok}
\DeclareMathOperator{\typ}{typ}
\DeclarePairedDelimiter{\abs}{\lvert}{\rvert}
\DeclarePairedDelimiterX{\Set}[2]{\{}{\}}{ #1 \mathchoice{\:}{\:}{\,}{\,}\delimsize\vert\allowbreak\mathchoice{\:}{\:}{\,}{\,}\mathopen{} #2 }
\DeclarePairedDelimiterX{\Seq}[2]{\langle}{\rangle}{ #1 \mathchoice{\:}{\:}{\,}{\,}\delimsize\vert\allowbreak\mathchoice{\:}{\:}{\,}{\,}\mathopen{} #2 }
\theoremstyle{plain}
\newtheorem{theorem}{Theorem}[section]
\newtheorem{proposition}[theorem]{Proposition}
\newtheorem{lemma}[theorem]{Lemma}
\newtheorem{corollary}[theorem]{Corollary}
\newtheorem{claim}{Claim}
\theoremstyle{definition}
\newtheorem{definition}[theorem]{Definition}
\newtheorem{question}[theorem]{Question}
\theoremstyle{remark}
\newtheorem{remark}[theorem]{Remark}
\newtheorem{notation}[theorem]{Notation}
\newtheorem{examples}[theorem]{Examples}
\begin{document}
\title{Properties preserved by classes of Chu transforms}
\author{Francesco Parente}
\address{Graduate School of System Informatics\\Kobe University\\1-1 Rokkodai-cho Nada-ku\\Kobe 657-8501\\Japan}
\thanks{The author is an International Research Fellow of the Japan Society for the Promotion of Science}
\thanks{Mirna D{\v z}amonja was a co-author of the accepted version of this paper. She withdrew her participation as an author due to her perception of the origins of the research project and the interference in the research procedure by the originators of the project. She apologizes to the Israel Journal of Mathematics for withdrawing her name as an author of the paper, especially at the moment when she strongly feels her support to Israel}
\begin{abstract}
Chu spaces and Chu transforms were first investigated in category theory by Barr and Chu in 1979. In 2000 van Benthem shifted to the model-theoretic point of view by isolating a class of infinitary two-sorted properties, the flow formulas, which are preserved by all Chu transforms. D{\v z}amonja and V{\"a}{\"a}n{\"a}nen in 2021 considered a special kind of Chu transforms, satisfying a density condition. These authors used dense Chu transforms to compare abstract logics, in particular showing that such transforms preserve compactness.

This motivates the problem of characterizing which properties are preserved by dense Chu transforms. We solve this problem by isolating the inconsistency-flow formulas, which are flow formulas with an added predicate to express inconsistency. Our main result characterizes the first-order inconsistency-flow formulas exactly as those first-order properties preserved by dense Chu transforms. Furthermore, we consider several instances of Chu transforms in concrete situations such as topological spaces, graphs and ultrafilters. In particular, we prove that Chu transforms between ultrafilters correspond to the Rudin-Keisler ordering.
\end{abstract}

\maketitle

\section{Introduction}

The Chu construction was introduced in 1979 by Michael Barr~\cite{MR0550878} and his student Po-Hsiang Chu~\cite[Appendix]{MR0550878}, who used it to produce examples of $*$-autonomous categories. When applied to the category of sets, this construction gives rise to Chu spaces and Chu transforms, which are quite a general notion that encompasses concepts such as continuous functions between topological spaces. Johan van Benthem, in \cite{MR1830523}, shifted the perspective on Chu spaces from the categorical point of view to the model-theoretic one. In this paper we extend that perspective, answering some of his questions and introducing new paradigms.

More precisely, van Benthem introduced a class of two-sorted infinitary formulas, called the \emph{flow formulas}, and proved a preservation theorem for the first-order fragment of this class under Chu transforms. In \cite[Definition~2.2]{MR4357456} D{\v z}amonja and Jouko V{\"a}{\"a}n{\"a}nen defined \emph{dense} Chu transforms and used them to compare abstract logics by a sublogic relation. As a result, they showed that dense Chu transforms between abstract logics preserve compactness. Dense Chu transforms are relevant to other contexts as well: for example, Chu transforms between partially ordered sets automatically satisfy the density condition. This motivates the problem of characterizing which properties are preserved by dense Chu transforms.

We solve this problem in our main Theorem~\ref{theorem:preservationdense}. This is done by isolating the class of \emph{inconsistency-flow} formulas, which are flow formulas with the added capability of expressing inconsistency in the second sort. As a consequence of Theorem~\ref{theorem:preservationdense}, we are able to characterize the properties preserved by the sublogic relation from \cite[Definition~2.2]{MR4357456}, at least for their first-order fragment.

Going back to van Benthem, in \cite[\S 7.4]{MR1830523} he suggested the possibility of obtaining further preservation results, notably for surjective Chu transforms. We show that such a result indeed holds, as a consequence of a general interpolation theorem by Martin Otto~\cite[Theorem~1]{MR1814123}. On a related note, Michael Gemignani~\cite{MR0235505} posed the problem of characterizing topological properties which are preserved by passing to a coarser topology. In Theorem~\ref{th:regressive=projective-flow} we obtain such a characterization by formulating Gemignani's problem in the language of Chu spaces and then applying Otto's theorem once again.

One of the main tenets of this paper is that Chu spaces and Chu transforms capture a variety of concepts in mathematics. For example, Chu transforms between partially ordered sets already appeared in the work of Micha{\"e}l Benado~\cite{MR0028817} in 1949, under the name ``connexions monotones d'esp{\`e}ce mixte''. In this paper we show that Chu transforms also encapsulate the Rudin-Keisler ordering between ultrafilters (Theorem~\ref{th:RudinKeislerorder}) and this through an argument that generalizes a 1956 proof by Walter Rudin~\cite{MR0080902}. For more recent applications, we show that Chu transforms correspond to strictly continuous functions between graphs (Proposition~\ref{proposition:graphchu}), as well as adjointable maps between orthosets with $0$ (Proposition~\ref{proposition:ortho}).

The presentation of our results is organized as follows. In Section~\ref{section:chu} we introduce the central notions of Chu spaces and Chu transforms and discuss some of their basic properties. Section~\ref{sec:Benthem} contains our preservation result for surjective Chu transforms, together with the solution to Gemignani's problem. Section~\ref{sec:inconflow} is dedicated to our main Theorem~\ref{theorem:preservationdense} and its role in the preservation of compactness properties of abstract logics and topological spaces. Finally, our results on ultrafilters and graphs are presented in Section~\ref{sec:ultrafilters} and Section~\ref{section:graphs}, respectively.

\section{Chu spaces and Chu transforms: definitions and examples}\label{section:chu}

The original definition of a Chu space comes from category theory and relates two possibly distinct sorts of objects by a binary relation.

\begin{definition}\label{def:Chu}
A ($2$-valued) \emph{Chu space} is a triple $\langle X,r,A\rangle$ with $r\subseteq X\times A$. Elements of $X$ are called \emph{points} and elements of $A$ are called \emph{states}.\footnote{In the literature, Chu spaces have sometimes been denoted by $\langle V,V',v\rangle$ or by $\langle A,r,X\rangle$. We adopt the notation which feels most natural to represent the examples we have in mind.}
\end{definition}

\begin{examples}\leavevmode
\begin{enumerate}
\item A typical and motivating example of a Chu space is a topological space, discussed in \S\ref{subsection:top}.
\item Another important example are abstract logics, introduced in \S\ref{subsection:log}.
\item In \S\ref{sec:ultrafilters} we discuss partially ordered sets, represented as Chu spaces $\langle P,\le,P \rangle$.
\item Perhaps more surprisingly, another example is obtained in the world of ultrafilters, simply by coding an ultrafilter $U$ as a Chu space $\langle U,\supseteq,U\rangle$. This seemingly trivial representation gives us a correspondence with the Rudin-Keisler ordering, as we show in Theorem~\ref{th:RudinKeislerorder}.
\item Two further examples are treated in \S\ref{section:graphs}, namely graphs and orthosets with~$0$.
\end{enumerate}
\end{examples}

\begin{notation}
We shall keep the convention that for a Chu space $\langle X,r,A\rangle$, the elements of $X$ are denoted by letters $x,y,z$ etc.\ and the elements of $A$ are denoted by $a,b,c$ etc. 
\end{notation}

The following separation properties seem to have first been considered by Barr~\cite[\S 6]{MR1132146}, under a slightly different terminology.

\begin{definition}
A Chu space $\langle X,r,A\rangle$ is:
\begin{itemize}
\item \emph{separated} if $\forall x\forall y\bigl(x=y\iff\forall a(\langle x,a\rangle\in r\iff\langle y,a\rangle\in r)\bigr)$,
\item \emph{extensional} if $\forall a\forall b\bigl(a=b\iff\forall x(\langle x,a\rangle\in r\iff\langle x,b\rangle\in r)\bigr)$.
\end{itemize}
\end{definition}

Although it is not used in this paper, we note that since the defining properties of separativity and extensionality are given through equivalence relations, every Chu space has a naturally defined separated and extensional quotient, which one could call the \emph{biextensional collapse}.

\begin{definition}\label{def:negconjdisj}
Let $\kappa$ be a cardinal. A Chu space $\langle X,r,A\rangle$:
\begin{itemize}
\item \emph{admits complements} if for every $a\in A$ there exists $b\in A$ such that for all $x\in X$ we have $\langle x,a\rangle\in r\iff\langle x,b\rangle\notin r$;
\item \emph{admits ${<}\kappa$-intersections} if for every $B\subseteq A$ with $\abs{B}<\kappa$ there exists $c\in A$ such that for all $x\in X$
we have $\langle x,c\rangle\in r\iff\forall b\in B(\langle x,b\rangle\in r)$;
\item \emph{admits ${<}\kappa$-unions} if for every $B\subseteq A$ with $\abs{B}<\kappa$ there exists $d\in A$ such that for all $x\in X$ we have $\langle x,d\rangle\in r\iff\exists c\in B(\langle x,c\rangle\in r)$.
\end{itemize}
\end{definition}

We define one of the central notions of this paper.

\begin{definition}\label{def:chut}
A \emph{Chu transform} from a Chu space $\langle X,r,A\rangle$ to a Chu space $\langle Y,s,B\rangle$ consists of two functions $f\colon X\to Y$ and $g\colon B\to A$ which are \emph{adjoint} in the sense that for every $x\in X$ and every $b\in B$
\[
\langle x,g(b)\rangle\in r\iff\langle f(x),b\rangle\in s.
\]
\end{definition}

\begin{remark}
The category of Chu spaces is related to both the Dialectica-like category of Valeria de Paiva~\cite{MR1031571} and the generalized Galois-Tukey connections\footnote{viewed as a category} of Peter Vojt{\'a}{\v s}~\cite{MR1234291}. This in the sense that they all have the same objects. Moreover, a pair of functions is a Chu transform if and only if it is both a Dialectica morphism and a generalized Galois-Tukey connection. For further details, we refer the reader to the comparative analyses of Andreas Blass~\cite{MR1356008} and de Paiva~\cite{zbMATH05171536}.
\end{remark}

In the nineties, Jon Barwise and Jerry Seligman~\cite{MR1472482} rediscovered Chu spaces and Chu transforms as the fundamental theory that underlies information flow. More precisely, a \emph{classification} $\bm{A}$, in the sense of \cite[Definition~4.1]{MR1472482}, is just a Chu space $\langle\tok(\bm{A}),\vDash_{\bm{A}},\typ(\bm{A})\rangle$, where $\tok(\bm{A})$ is the set of tokens of $\bm{A}$, $\typ(\bm{A})$ is the set of types of $\bm{A}$, and $a\vDash_{\bm{A}}\alpha$ holds if and only if the token $a$ is of type $\alpha$. In addition to that, given classifications $\bm{A}$ and $\bm{B}$ together with functions $f\hat{\mkern6mu}\colon\typ(\bm{A})\to\typ(\bm{B})$ and $f\check{\mkern6mu}\colon\tok(\bm{B})\to\tok(\bm{A})$, we have that $(f\hat{\mkern6mu},f\check{\mkern6mu})$ is an \emph{infomorphism} from $\bm{A}$ to $\bm{B}$, in the sense of \cite[Definition~4.8]{MR1472482}, if and only if $(f\check{\mkern6mu},f\hat{\mkern6mu})$ is a Chu transform from $\langle\tok(\bm{B}),\vDash_{\bm{B}},\typ(\bm{B})\rangle$ to $\langle\tok(\bm{A}),\vDash_{\bm{A}},\typ(\bm{A})\rangle$.

The following special kinds of Chu transforms have been considered in the literature.

\begin{definition}\label{def:surj}
A Chu transform $(f,g)$ from $\langle X,r,A\rangle$ to $\langle Y,s,B\rangle$ is
\begin{itemize}
\item \emph{dense} iff for every $b\in B$, if there exists $y\in Y$ such that $\langle y,b\rangle\in s$, then there exists $x\in X$ such that $\langle f(x),b\rangle\in s$;
\item \emph{surjective} iff $f$ is a surjective function;
\item a \emph{restriction in the second sort} iff $X=Y$, $B\subseteq A$, $f=\mathrm{id}_X$, and $g=\mathrm{id}_B$.
\end{itemize}
\end{definition}

In addition to \cite{MR1830523}, where they were first considered, surjective Chu transforms were defined by Marta Garc{\'\i}a-Matos and V{\"a}{\"a}n{\"a}nen~\cite{MR2134728} in the context of abstract logics.\footnote{The authors of \cite{MR2134728} do not use the terminology of Chu transforms, but rather define the specific instance of a Chu transform between abstract logics which gives the sublogic relation, \cite[Definition~2.2]{MR2134728}.} Restrictions in the second sort were used by Solomon Feferman in \cite{feferman:johan99} and Otto in \cite{MR1814123}. Dense Chu transforms were introduced more recently by D{\v z}amonja and V{\"a}{\"a}n{\"a}nen~\cite[Definition~2.2]{MR4357456}.

From Definition~\ref{def:surj}, it follows that every surjective Chu transform is dense and that every restriction in the second sort is surjective. Also notice that, in the case of restrictions in the second sort, the adjointness condition of Definition~\ref{def:chut} comes down to $\langle x,b\rangle\in r\iff\langle x,b\rangle\in s$, which is just saying that $s=r\cap (Y\times B)$.

\subsection{Topological spaces}\label{subsection:top}

Topological spaces can be encoded as Chu spaces. We present a slightly different angle on this by considering the topological systems defined by Steven Vickers.

\begin{definition}[{Vickers~\cite[Definition~5.1.1]{MR1002193}}]\label{def:vi}
A \emph{topological system} is a Chu space which admits ${<}\aleph_0$-intersections and admits ${<}\kappa$-unions for every $\kappa$.
\end{definition}

According to Definition~\ref{def:vi}, if $\langle X,\tau(X)\rangle$ is a topological space then $\langle X,\in,\tau(X)\rangle$ is an extensional topological system. Let us note that $\langle X,\tau(X)\rangle$ is $T_0$ if and only if $\langle X,\in,\tau(X)\rangle$ is separated. For topological systems, admitting ${<}\kappa$-intersections means that the intersection of fewer than $\kappa$ many open sets is open, which true by definition if $\kappa=\aleph_0$ and true for larger $\kappa$ for the ${<}\kappa$-open spaces defined by Roman Sikorski in \cite{MR0040643} and considered by David Buhagiar and D{\v z}amonja in \cite{MR4646733}.

If $\mathcal{B}(X)$ is a basis for the topology of $X$, then $\langle X,\in,\mathcal{B}(X)\rangle$ is also a Chu space. For this space, admitting complements means that $\mathcal{B}(X)$ is closed under set-theoretic complements and hence $\mathcal{B}(X)$ consists of clopen sets and $X$ is zero-dimensional.

The next proposition shows that Chu transforms between topological spaces naturally correspond to continuous functions.

\begin{proposition}[{Vickers~\cite[Proposition~5.3.1]{MR1002193}}]\label{proposition:vickers}
Let $\langle X,\tau(X)\rangle$ and $\langle Y,\tau(Y)\rangle$ be topological spaces. For a pair of functions $f\colon X\to Y$ and $g\colon\tau(Y)\to\tau(X)$, the following conditions are equivalent:
\begin{itemize}
\item $(f,g)$ is a Chu transform from $\langle X,\in,\tau(X)\rangle$ to $\langle Y,\in,\tau(Y)\rangle$;
\item $f$ is continuous and for every $b\in\tau(Y)$, $f^{-1}[b]=g(b)$.
\end{itemize}
\end{proposition}

It is now easy to see that, for topological spaces, dense Chu transforms correspond to continuous functions with dense image, surjective Chu transforms correspond to surjective continuous functions, and a restriction in the second sort is just a coarsening of the topology.

\subsection{Abstract logics}\label{subsection:log}

The concept of \emph{abstract logic} can be traced back to the work of Andrzej Mostowski~\cite{MR0089816}, Per Lindstr{\"o}m~\cite{MR0244013}, and Barwise~\cite{MR0376337}. For the purpose of this paper, an abstract logic is simply a Chu space $\langle M,\vDash,S\rangle$, where $M$ are the models, $S$ are the sentences, and $\vDash$ is the satisfaction relation between models and sentences.

An abstract logic $\langle M,\vDash,S\rangle$ is separated if and only if no two different models in $M$ are elementarily equivalent, and it is extensional if and only if no two different sentences in $S$ are logically equivalent. To admit complements means that $S$ is closed under negation and admitting ${<}\kappa$-intersections (unions) means that $S$ is closed under ${<}\kappa$-conjunctions (disjunctions).

A key observation is that a Chu transform $(f,g)$ from an abstract logic $\langle M,\vDash, S\rangle$ to an abstract logic $\langle M',\vDash', S'\rangle$ is dense if and only if $g$ preserves satisfiability, in the sense that if $\sigma'\in S'$ is satisfiable then $g(\sigma')$ must also be satisfiable. Motivated by this observation, dense Chu transforms can be used to define a sublogic relation between abstract logics.

\begin{definition}[{D{\v z}amonja and V{\"a}{\"a}n{\"a}nen~\cite[Definition~2.2]{MR4357456}}]
Let $\langle M,\vDash, S\rangle$ and $\langle M',\vDash', S'\rangle$ be abstract logics. We say that $\langle M',\vDash', S'\rangle$ is a \emph{sublogic} of $\langle M,\vDash,S\rangle$, in symbols $\langle M',\vDash', S'\rangle\le\langle M,\vDash, S\rangle$, if there exists a dense Chu transform from $\langle M,\vDash, S\rangle$ to $\langle M',\vDash', S'\rangle$.
\end{definition}

Further properties of abstract logics will be discussed in \S\ref{sec:inconflow}.

\section{Flow formulas and preservation under certain types of Chu transforms}\label{sec:Benthem}

According to van Benthem~\cite{MR1830523}, each Chu space can be naturally viewed as a two-sorted relational structure, with one sort consisting of points and the other sort consisting of states. Without loss of generality, we shall assume the two sorts to be disjoint.

The language consists of two-sorted infinitary formulas, with variables $v^\mathsf{p}$ corresponding to points, variables $v^\mathsf{s}$ corresponding to states, and a binary relation symbol $R$ meant to be interpreted as the relation $r$ between sorts. Boldface letters, such as $\bm{v}^\mathsf{p}$ or $\bm{v}^\mathsf{s}$, syntactically denote (possibly infinite) tuples.

\begin{definition}
The $\mathcal{L}^\mathrm{II}_{\infty,\infty}$-formulas are defined recursively as follows:
\begin{itemize}
\item $(v^\mathsf{p}_0=v^\mathsf{p}_1)$, $(v^\mathsf{s}_0=v^\mathsf{s}_1)$ and $R(v^\mathsf{p},v^\mathsf{s})$ are atomic $\mathcal{L}^\mathrm{II}_{\infty,\infty}$-formulas;
\item the negation of an $\mathcal{L}^\mathrm{II}_{\infty,\infty}$-formula is an $\mathcal{L}^\mathrm{II}_{\infty,\infty}$-formula;
\item infinitary conjunctions and disjunctions of $\mathcal{L}^\mathrm{II}_{\infty,\infty}$-formulas are $\mathcal{L}^\mathrm{II}_{\infty,\infty}$-for\-mulas;
\item if $\varphi$ is an $\mathcal{L}^\mathrm{II}_{\infty,\infty}$-formula, then $\exists\bm{v}^\mathsf{p}\varphi$, $\forall\bm{v}^\mathsf{p}\varphi$, $\exists\bm{v}^\mathsf{s}\varphi$ and $\forall\bm{v}^\mathsf{s}\varphi$ are $\mathcal{L}^\mathrm{II}_{\infty,\infty}$-formulas.
\end{itemize}

Let $\kappa\le\lambda$ be infinite cardinals. The fragment $\mathcal{L}^\mathrm{II}_{\lambda,\kappa}$ consists of formulas which contain only conjunctions and disjunctions of size $<\lambda$ and quantification over tuples of length $<\kappa$.
\end{definition}

The interpretation of the satisfaction relation for this two-sorted language is done analogously to the standard definition for the single-sorted $\mathcal{L}_{\infty,\infty}$ language; see the book by Herbert Enderton~\cite[\S 4.3]{MR1801397} for further details on many-sorted logic. In particular, each two-sorted structure can be converted into a single-sorted structure with additional relation symbols to keep track of the sorts. In this way, usual model-theoretic techniques such as compactness and elementary extensions can be applied to the $\mathcal{L}^\mathrm{II}_{\omega,\omega}$ fragment.

\begin{definition}\label{def:preservation}
We say that an $\mathcal{L}^\mathrm{II}_{\infty,\infty}$-formula $\varphi(\bm{v}^\mathsf{p},\bm{v}^\mathsf{s})$ is \emph{preserved under Chu transforms} satisfying some property $P$ if, whenever $(f,g)$ is a Chu transform from $\langle X,r,A\rangle$ to $\langle Y,s,B\rangle$ and $(f,g)$ satisfies $P$, then for all $\bm{x}\in X$ and $\bm{b}\in B$ we have
\[
\langle X,r,A\rangle\models\varphi[\bm{x},g(\bm{b})]\implies\langle Y,s,B\rangle\models\varphi[f(\bm{x}),\bm{b}].
\]
The notation $f(\bm{x})$ denotes the tuple obtained by applying $f$ pointwise to $\bm{x}$, and similarly for $g(\bm{b})$.
\end{definition}

\subsection{Flow formulas and preservation under Chu transforms}

Flow formulas form a subclass of $\mathcal{L}^\mathrm{II}_{\infty,\infty}$ obtained by restricting equality, negation and quantification. Here are the formal definition and the first results of \cite{MR1830523}.

\bigskip

\begin{definition}\label{definition:flow}
The \emph{flow formulas} are defined recursively as follows:
\begin{itemize}
\item $(v^\mathsf{p}_0=v^\mathsf{p}_1)$, $\lnot(v^\mathsf{s}_0=v^\mathsf{s}_1)$, $R(v^\mathsf{p},v^\mathsf{s})$, and $\lnot R(v^\mathsf{p},v^\mathsf{s})$
are atomic flow formulas;
\item infinitary conjunctions and disjunctions of flow formulas are flow formulas;
\item if $\varphi$ is a flow formula, then $\exists\bm{v}^\mathsf{p}\varphi$ and $\forall\bm{v}^\mathsf{s}\varphi$ are flow formulas.
\end{itemize}
\end{definition}

\begin{remark}\label{why-not-equal}
We explain the special treatment of the equality and negation in Definition~\ref{definition:flow}. Namely, the $\mathcal{L}^\mathrm{II}_{\infty,\infty}$-atomic formulas of the form $\lnot(v^\mathsf{p}_0=v^\mathsf{p}_1)$ and $(v^\mathsf{s}_0=v^\mathsf{s}_1)$ are not included in the class of flow formulas. The reason is that the main purpose of flow formulas is to characterize the properties of Chu spaces which are preserved by Chu transforms, as in Proposition~\ref{proposition:vb}. The property  $\lnot(v^\mathsf{p}_0=v^\mathsf{p}_1)$ is not preserved by a Chu transform $(f,g)$, unless $f$ is injective. Analogously, the property $(v^\mathsf{s}_0=v^\mathsf{s}_1)$ is not preserved by $(f,g)$, unless $g$ is injective. The original setup by van Benthem did not include any equalities.
\end{remark}

The following proposition is a generalization of \cite[Proposition~3.1]{MR1830523}, where the result is established for first-order flow formulas. Since our class of formulas is larger, we provide a detailed proof of the proposition, for completeness.

\begin{proposition}\label{proposition:vb}
All flow formulas are preserved under Chu transforms.
\end{proposition}
\begin{proof}
Suppose that $(f,g)$ is a Chu transform between Chu spaces $\langle X,r,A\rangle $ and $\langle Y,s,B\rangle$. Let $\varphi$ be any flow formula; we use induction on the complexity of $\varphi$. 

If $x_0,x_1\in X$ are such that $x_0=x_1$, then obviously $f(x_0)=f(x_1)$, and similarly if $b_0,b_1\in B$ satisfy $g(b_0)\neq g(b_1)$, then $b_0\neq b_1$.
If $\varphi$ is $R(v^\mathsf{p},v^\mathsf{s})$ and $\langle X,r,A\rangle\models\varphi[x,g(b)]$ for some $x\in X$ and $b\in B$, then the adjointness condition gives $\langle Y,s,B\rangle\models\varphi[f(x),b]$. Similarly, since the adjointness condition is given in the form of an equivalence, we obtain that the flow formulas of the form $\lnot R(v^\mathsf{p},v^\mathsf{s})$ are preserved. 

A straightforward argument applies to infinitary conjunctions and disjunctions.

If $\langle X,r,A\rangle\models\exists\bm{v}^\mathsf{p}\varphi[\bm{v}^\mathsf{p},\bm{x},g(\bm{b})]$ for some $\bm{x}$ in $X$ and $\bm{b}$ in $B$, then there exists a tuple $\bm{x}'$ in $X$ such that $\langle X,r,A\rangle\models\varphi[\bm{x}',\bm{x},g(\bm{b})]$. This implies inductively that $\langle Y,s,B\rangle\models\varphi[f(\bm{x}'),f(\bm{x}),\bm{b}]$ and finally $\langle Y,s,B\rangle\models\exists\bm{v}^\mathsf{p}\varphi[\bm{v}^\mathsf{p},f(\bm{x}),\bm{b}]$. The case for $\forall\bm{v}^\mathsf{s}\varphi$ is similar.
\end{proof}

Using his framework without equality, van Benthem shows as in Theorem~\ref{the-vb} that Proposition~\ref{proposition:vb} is optimal, at least for the first-order fragment. Feferman~\cite[Theorem~1]{feferman:johan99} gives another proof of Theorem~\ref{the-vb}, using the same framework with equality as in the present paper.

\begin{theorem}[{van Benthem~\cite[Theorem~6.1]{MR1830523}}]\label{the-vb}
Let $\varphi(\bm{v}^\mathsf{p},\bm{v}^\mathsf{s})$ and $\psi(\bm{v}^\mathsf{p},\bm{v}^\mathsf{s})$ be $\mathcal{L}^\mathrm{II}_{\omega,\omega}$-formulas. Then the following are equivalent:
\begin{enumerate}
\item If $(f,g)$ is a Chu transform from $\langle X,r,A\rangle$ to $\langle Y,s,B\rangle$, then for every $\bm{x}\in X$ and $\bm{b}\in B$, we have
\[
\langle X,r,A\rangle\models\varphi[\bm{x},g(\bm{b})]\implies\langle Y,s,B\rangle\models\psi[f(\bm{x}),\bm{b}].
\]
\item There exists a first-order flow formula $\theta(\bm{v}^\mathsf{p},\bm{v}^\mathsf{s})$ such that
\[
\vdash\forall\bm{v}^\mathsf{p}\forall\bm{v}^\mathsf{s}\bigl((\varphi\rightarrow\theta)\land(\theta\rightarrow\psi)\bigr).
\]
\end{enumerate}
\end{theorem}

Setting $\varphi$ equal to $\psi$ in the above theorem gives that an $\mathcal{L}^\mathrm{II}_{\omega,\omega}$-formula is preserved under Chu transforms if and only if it is logically equivalent to a first-order flow formula.

\subsection{Universal-flow formulas and preservation under surjective Chu transforms}

In this subsection we focus on surjective Chu transforms, which were considered in \cite{MR1830523}. Discussing further ramifications of his Theorem~\ref{the-vb} above, van Benthem makes an informal remark \cite[\S 7.4]{MR1830523}: ``if we know that the $f$-map in a Chu transform is \emph{surjective}, then we can add \emph{universal object quantifiers} $\forall a$ in the construction of flow formulas''. We formalize his remark by introducing the class of universal-flow formulas (Definition~\ref{def:universal-flow}) and then proving the corresponding preservation result (Proposition~\ref{prop:surjective}).

\begin{definition}\label{def:universal-flow}
The \emph{universal-flow formulas} are defined recursively as follows:
\begin{itemize}
\item $(v^\mathsf{p}_0=v^\mathsf{p}_1)$, $\lnot(v^\mathsf{s}_0=v^\mathsf{s}_1)$, $R(v^\mathsf{p},v^\mathsf{s})$, and $\lnot R(v^\mathsf{p},v^\mathsf{s})$ are atomic universal-flow formulas;
\item infinitary conjunctions and disjunctions of universal-flow formulas are uni\-versal-flow formulas;
\item if $\varphi$ is a universal-flow formula, then $\exists\bm{v}^\mathsf{p}\varphi$, $\forall\bm{v}^\mathsf{p}\varphi$, and $\forall\bm{v}^\mathsf{s}\varphi$ are universal-flow formulas.
\end{itemize}
\end{definition}

\begin{proposition}\label{prop:surjective}
All universal-flow formulas are preserved under surjective Chu transforms.
\end{proposition}
\begin{proof}
We follow the proof of Proposition~\ref{proposition:vb}, using the induction on the complexity of a given universal-flow formula $\varphi$. The new case is to suppose that  $\langle X,r,A\rangle\models\forall\bm{v}^\mathsf{p}\varphi[\bm{v}^\mathsf{p},\bm{x},g(\bm{b})]$ for some $\bm{x}$ in $X$ and $\bm{b}$ in $B$. If $\bm{y}$ is in $Y$, by surjectivity there exists $\bm{x}'\in X$ such that $f(\bm{x}')=\bm{y}$. We then have that $\langle X,r,A\rangle\models\varphi[\bm{x}',\bm{x},g(\bm{b})]$. This implies inductively that $\langle Y,s,B\rangle\models\varphi[\bm{y},f(\bm{x}),\bm{b}]$, from which we obtain $\langle Y,s,B\rangle\models\forall\bm{v}^\mathsf{p}\varphi[\bm{v}^\mathsf{p},f(\bm{x}),\bm{b}]$.
\end{proof}

In analogy with Theorem~\ref{the-vb}, we aim to prove that if an $\mathcal{L}^\mathrm{II}_{\omega,\omega}$-formula is preserved by all surjective Chu transforms, then it must be logically equivalent to a first-order universal-flow formula. Our proof is an application of a general interpolation theorem by Otto~\cite{MR1814123}, which we state as Theorem~\ref{th-otto}.

Otto's theorem is in turn a strengthening of Lyndon's interpolation theorem \cite{MR0106825} and uses the same notation of positive and negative occurrences of a symbol in a formula. For that definition, we do not allow connectives $\rightarrow$ and $\leftrightarrow$, which anyway do not figure in our definition of $\mathcal{L}^\mathrm{II}_{\infty,\infty}$-formulas. For a symbol $S$ of a language $\mathcal{L}$ and a formula $\varphi$, we say that $S$ occurs \emph{positively} in $\varphi$ if $S$ has an occurrence within $\varphi$ which is in the scope of an even number of negation symbols. The symbol $S$ occurs \emph{negatively} in $\varphi$ if $S$ has an occurrence within $\varphi$ which is in the scope of an odd number of negation symbols. Note that a symbol can occur both positively and negatively in a formula.\footnote{We can see that the connectives $\rightarrow$ and $\leftrightarrow$ are excluded since they introduce hidden negation symbols that can change the parity of an occurrence.}

Suppose that $\mathbb{U}$ is a tuple of designated unary predicates and $\varphi$ a formula. We say that $\varphi$ is \emph{$\mathbb{U}$-relativized} if each
quantifier in $\varphi$ is bounded by a predicate in $\mathbb{U}$, namely it is of the form $\exists x(U(x)\land\dots)$ or $\forall x(\lnot U(x)\lor\dots)$, for
some $U\in\mathbb{U}$.

\begin{theorem}[{Otto~\cite[Theorem~1]{MR1814123}}]\label{th-otto}
Let $\mathcal{L}$ be a finite relational language with $\top$ and $\bot$ as atomic formulas, and $\mathbb{U}$ a tuple of designated unary predicates. Suppose that $\varphi$ and $\psi$ are two $\mathbb{U}$-relativized first-order formulas such that $\varphi\vdash\psi$. Then there is a $\mathbb{U}$-relativized first-order formula $\chi$ such that:
\begin{enumerate}
\item Every free variable of $\chi$ is free in both $\varphi$ and $\psi$.
\item\label{otto:pos} Every relation symbol of $\mathcal{L}$ which occurs positively in $\chi$, occurs positively in both $\varphi$ and $\psi$.
\item\label{otto:neg} Every relation symbol of $\mathcal{L}$ which occurs negatively in $\chi$, occurs negatively in both $\varphi$ and $\psi$.
\item $\varphi\vdash\chi$ and $\chi\vdash\psi$.
\end{enumerate}

The statement of the theorem has two readings, one for first-order logic with equality and another for first-order logic without equality.
\end{theorem}

Applying the version of Theorem~\ref{th-otto} without equality, Otto~\cite[Theorem~14]{MR1814123} gave an alternative proof of van Benthem's Theorem~\ref{the-vb}. From that application, we shall now deduce a preservation theorem for first-order universal-flow formulas.

\begin{corollary}\label{the-universalpreservation}
Suppose $\varphi(\bm{v}^\mathsf{p},\bm{v}^\mathsf{s})$ is an $\mathcal{L}^\mathrm{II}_{\omega,\omega}$-formula which is preserved under surjective Chu transforms. Then there exists a first-order universal-flow formula $\theta(\bm{v}^\mathsf{p},\bm{v}^\mathsf{s})$ such that $\varphi$ and $\theta$ are logically equivalent and, if $R$ occurs positively (negatively) in $\theta$, then $R$ occurs positively (negatively) in $\varphi$.
\end{corollary}
\begin{proof}
As a preliminary step, as noted in Remark~\ref{why-not-equal}, 
$\mathcal{L}^\mathrm{II}_{\omega,\omega}$-atomic formulas of the type $\lnot(v^\mathsf{p}_0=v^\mathsf{p}_1)$ are not preserved unless $f$ is injective, which is not required for surjective Chu transforms. Hence, up to logical equivalence, $\varphi$ does not contain any occurrences of such atomic formulas. Similarly for atomic formulas of the type $(v^\mathsf{s}_0=v^\mathsf{s}_1)$, which are only preserved if $g$ is injective.

Next, following the proof of \cite[Theorem~14]{MR1814123}, let $\mathbb{U}=\langle U,V\rangle$ consist of two unary predicates for the points sort and the states sort, respectively. Let $\varphi^\mathbb{U}$ be obtained from $\varphi$ by relativizing each quantifier accordingly. Otto shows that, if $\varphi$ is preserved under restrictions in the second sort (which is certainly true for our $\varphi$), then $\varphi^\mathbb{U}$ is logically equivalent to a $\mathbb{U}$-relativized formula $\chi$ in which $V$ occurs only negatively, and each free variable of $\chi$ is free in $\varphi^\mathbb{U}$.

The fact that $\chi$ is $\mathbb{U}$-relativized implies that $\chi$ has a rewriting as an $\mathcal{L}^\mathrm{II}_{\omega,\omega}$-formula: let us call this rewriting $\theta(\bm{v}^\mathsf{p},\bm{v}^\mathsf{s})$. Then, the fact that $V$ occurs only negatively in $\chi$ means that $\theta$ does not admit existential quantification over the sort corresponding to $V$, which is exactly how we defined universal-flow formulas. Finally, the additional syntactic information on positive and negative occurrences of $R$ in $\theta$ can be deduced from points \eqref{otto:pos} and \eqref{otto:neg} of Theorem~\ref{th-otto}, as in the proof of \cite[Proposition~11]{MR1814123}.
\end{proof}

\subsection{Regressive-flow formulas and preservation under restrictions in the second sort}

Gemignani~\cite{MR0235505} considered topological properties preserved by continuous functions. He did not define what a ``property'' is but, for our purpose, we shall take it to mean any $\mathcal{L}^\mathrm{II}_{\infty,\infty}$-sentence. 

We start by recalling Gemignani's definition of a regressive property and giving its generalization to the language of Chu spaces. Notice that, when specialized to topological spaces, our definition coincides with Gemignani's.

\begin{definition}\leavevmode
\begin{enumerate}
\item (Gemignani~\cite{MR0235505}) A property $P$ is \emph{regressive} iff for every topological space $\langle X,\tau\rangle$ and every coarser topology $\tau'\subseteq\tau$ on $X$, if $\langle X,\tau\rangle$ satisfies $P$ then also $\langle X,\tau'\rangle$ satisfies $P$.
\item An $\mathcal{L}^\mathrm{II}_{\infty,\infty}$-sentence $\varphi$ is \emph{regressive} iff it is preserved under restrictions in the second sort.
\end{enumerate}
\end{definition}

In Theorem~\ref{th:regressive=projective-flow}, we shall characterize first-order regressive sentences as a class extending the flow formulas, as introduced in the following definition.

\begin{definition}
The \emph{regressive-flow formulas} are defined recursively as follows:
\begin{itemize}
\item all atomic $\mathcal{L}^\mathrm{II}_{\infty,\infty}$-formulas are atomic regressive-flow formulas;
\item infinitary conjunctions and disjunctions of regressive-flow formulas are re\-gressive-flow formulas;
\item if $\varphi$ is a regressive-flow formula, $\exists\bm{v}^\mathsf{p}\varphi$, $\forall\bm{v}^\mathsf{p}\varphi$, and $\forall\bm{v}^\mathsf{s}\varphi$ are regressive-flow formulas.
\end{itemize}
\end{definition}

We notice that the regressive-flow formulas are essentially the same as the universal-flow formulas of Definition~\ref{def:universal-flow}. The only difference is the use of equality and non-equality in the atomic case, which is now unrestricted.

\begin{proposition}\label{prop:preserve-projective-flow}
All regressive-flow formulas are preserved under restrictions in the second sort.
\end{proposition}
\begin{proof}
Since restrictions in the second sort are in particular surjective, the inductive step is the same as the proof of Proposition~\ref{prop:surjective}. Additionally, we only need to check that inequalities of the type $\lnot(v_0^\mathsf{p}= v_1^\mathsf{p})$ and equalities of the type $(v_0^\mathsf{s}= v_1^\mathsf{s})$ are preserved, but this clearly holds by $f$ and $g$ being identity functions and by the logical convention of $=$ being interpreted as the true equality.
\end{proof}

First-order regressive sentences can be characterized by the following theorem.

\begin{theorem}\label{th:regressive=projective-flow}
An $\mathcal{L}^\mathrm{II}_{\omega,\omega}$-sentence is regressive if and only if it is logically equivalent to a first-order regressive-flow sentence.
\end{theorem}
\begin{proof}
The right-to-left implication is a special case of Proposition~\ref{prop:preserve-projective-flow}. For the other implication, we can use the same argument as for Corollary~\ref{the-universalpreservation}. The only difference is that, since now both $f$ and $g$ are injective, inequalities of the form $\lnot(v^\mathsf{p}_0=v^\mathsf{p}_1)$ and equalities of the form $(v^\mathsf{s}_0=v^\mathsf{s}_1)$ are not ruled out from the atomic case.
\end{proof}

Gemignani~\cite[Proposition~1]{MR0235505} claimed, among other things, that if a regressive property is in addition preserved by continuous bijective functions, then it is preserved by all continuous surjective functions. However, this is false because, for example, the property of not being $T_0$ is regressive and preserved by continuous bijective functions, but is not preserved by continuous surjective functions.

\section{Preservation of compactness and inconsistency-flow formulas}\label{sec:inconflow}

It was shown in \cite{MR2134728} that surjective Chu transforms preserve compactness of abstract logics. Furthermore, in \cite{MR4357456} it was shown that dense Chu transforms preserve $\langle\kappa,\lambda\rangle$-compactness between abstract logics admitting complements and ${<}\kappa$-intersections. Motivated by these results, in this section we give a general preservation theorem for dense Chu transforms. We explain in detail how it applies to the preservation of compactness both in topology and in logic. In particular, we obtain the earlier results as corollaries.

We start with a general definition of compactness for Chu spaces, which follows the one by Paul Alexandroff and Paul Urysohn~\cite{MR0370492} of compactness for topological spaces.

\begin{definition}\label{def:compactness}
Let $\kappa\le\lambda$ be infinite cardinals. A Chu space is \emph{$\langle\kappa,\lambda\rangle$-compact} if it satisfies the following $\mathcal{L}^\mathrm{II}_{{(\lambda^{<\kappa})}^+,\lambda^+}$-sentence:
\begin{equation}\label{for:compactness}
\forall\Seq{v^\mathsf{s}_\alpha}{\alpha<\lambda}\biggl(\exists v^\mathsf{p}\bigwedge_{\alpha<\lambda}\lnot R(v^\mathsf{p},v^\mathsf{s}_\alpha)\ \lor\bigvee_{Z\in{[\lambda]}^{<\kappa}}\forall v^\mathsf{p}\bigvee_{\alpha\in Z}R(v^\mathsf{p},v^\mathsf{s}_\alpha)\biggr).
\end{equation}
We say that a Chu space is \emph{compact} if it is $\langle\aleph_0,\lambda\rangle$-compact for all infinite $\lambda$.
\end{definition}

We observe that Definition~\ref{def:compactness} encompasses both topological and logical compactness as special cases.

\begin{remark}
Let $\langle X,\tau(X)\rangle$ be a topological space; then the following are equivalent:
\begin{itemize}
\item the Chu space $\langle X,\in,\tau(X)\rangle$ is $\langle\kappa,\lambda\rangle$-compact,
\item every open cover of $X$ of cardinality $\le\lambda$ has a subcover of cardinality $<\kappa$.
\end{itemize}
Let $\langle M,\vDash,S\rangle$ be an abstract logic which is closed under negation; then the following are equivalent:
\begin{itemize}
\item the Chu space $\langle M,\vDash,S\rangle$ is $\langle\kappa,\lambda\rangle$-compact,
\item given a set of sentences $T$ of cardinality $\le\lambda$, if each subset of $T$ of cardinality $<\kappa$ has a model, then $T$ itself has a model.
\end{itemize}
The case of topological spaces is clear. For logics which are closed under negation, we first have to note that the corresponding Chu spaces admit complements: for every sentence $\sigma$, there is a negation $\lnot\sigma$ and any model satisfies exactly one of them. Then, to obtain the equivalence, take a sequence $\Seq{\sigma_\alpha}{\alpha<\lambda}$ of sentences and apply Definition~\ref{def:compactness} to  $\Seq{\lnot\sigma_\alpha}{\alpha<\lambda}$.
\end{remark}

The sentence in \eqref{for:compactness} is universal flow. Hence, as a consequence of Proposition~\ref{prop:surjective}, we obtain the following corollary.

\begin{corollary}\label{cor-compact}
Surjective Chu transforms preserve $\langle\kappa,\lambda\rangle$-compactness of Chu spaces.
\end{corollary}

Concentrating on the sublogic relation, Garc{\'\i}a-Matos and V{\"a}{\"a}n{\"a}nen proved the relevant instance of Corollary~\ref{cor-compact} in \cite[Lemma~2.3]{MR2134728}. Subsequently, D{\v z}amonja and V{\"a}{\"a}n{\"a}nen generalized that result to dense Chu transforms \cite[Corollary~2.6]{MR4357456}. This motivates us to characterize which formulas are preserved by dense Chu transforms and the relation of such formulas with preserving compactness. We naturally connect this with the notion of consistency.

First of all, note that flow formulas can already express that a state is ``consistent'', in symbols $\Con(v^\mathsf{s})$, as
\[
\Con(v^\mathsf{s})\equiv\exists v^\mathsf{p} R(v^\mathsf{p},v^\mathsf{s}).
\] 
In the case of an abstract logic, the formula $\Con[a]$ expresses the satisfiability of the sentence $a$, in the sense of $a$ having a model. In topological spaces, $\Con[a]$ simply means that the open set $a$ is non-empty. However, flow formulas cannot express inconsistency because formulas of the form $\lnot\exists\bm{v}^\mathsf{p}\varphi$ or $\forall\bm{v}^\mathsf{p}\psi$ are not in the class of flow formulas. So we shall define the class of ``inconsistency-flow'' formulas as the formulas which can express not only consistency, but also inconsistency. They are formally defined as follows.

\begin{definition}
The \emph{inconsistency-flow formulas} are defined recursively as follows:
\begin{itemize}
\item $(v^\mathsf{p}_0=v^\mathsf{p}_1)$, $\lnot(v^\mathsf{s}_0=v^\mathsf{s}_1)$, $R(v^\mathsf{p},v^\mathsf{s})$, $\lnot R(v^\mathsf{p},v^\mathsf{s})$, and $\lnot\exists v^\mathsf{p}R(v^\mathsf{p},v^\mathsf{s})$ are atomic inconsistency-flow formulas;
\item infinitary conjunctions and disjunctions of inconsistency-flow formulas are inconsistency-flow formulas;
\item if $\varphi$ is an inconsistency-flow formula, then $\exists\bm{v}^\mathsf{p}\varphi$ and $\forall\bm{v}^\mathsf{s}\varphi$ are inconsistency-flow formulas.
\end{itemize}
\end{definition}

As our aim is to prove that inconsistency-flow formulas are preserved by dense, in particular by surjective Chu transforms, we have defined the class of inconsistency-flow formulas as a subclass of the universal-flow formulas. This holds because universal-flow formulas allow quantification of the form $\forall\bm{v}^\mathsf{p}\varphi$ and so in particular $\forall v^\mathsf{p}\lnot R(v^\mathsf{p},v^\mathsf{s})$, giving us $\lnot\exists v^\mathsf{p}R(v^\mathsf{p},v^\mathsf{s})$.

\begin{proposition}\label{prop:dense}
All inconsistency-flow formulas are preserved under dense Chu transforms.
\end{proposition}
\begin{proof}
Suppose $\langle Y,s,B\rangle\models\exists v^\mathsf{p}R(v^\mathsf{p},b)$ for some $b\in B$, which means there exists $y\in Y$ such that $\langle y,b\rangle\in s$. By the density condition, there exists $x\in X$ such that $\langle f(x),b\rangle\in s$, which is equivalent to $\langle x,g(b)\rangle\in r$. This implies that $\langle X,r,A\rangle\models\exists v^\mathsf{p}R(v^\mathsf{p},g(b))$. The rest is similar to Proposition~\ref{proposition:vb}.
\end{proof}

\begin{examples}\leavevmode
\begin{enumerate}
\item For every infinite cardinal $\kappa$, the topological property of having a dense subset of cardinality $\le\kappa$ is expressible by the following $\mathcal{L}^\mathrm{II}_{\kappa^+,\kappa^+}$-inconsistency-flow sentence:
\[
\exists\Seq{v^\mathsf{p}_\alpha}{\alpha<\kappa}\forall v^\mathsf{s}\biggl(\lnot\Con(v^\mathsf{s})\lor\bigvee_{\alpha<\kappa}R(v^\mathsf{p}_\alpha,v^\mathsf{s})\biggr).
\]
\item The topological property of connectedness is expressible by the following $\mathcal{L}^\mathrm{II}_{\omega,\omega}$-inconsistency-flow sentence:
\begin{multline*}
\forall v^\mathsf{s}_0\forall v^\mathsf{s}_1\bigl(\lnot\Con(v^\mathsf{s}_0)\lor\lnot\Con(v^\mathsf{s}_1)\\ \lor\exists v^\mathsf{p}(R(v^\mathsf{p},v^\mathsf{s}_0)\land R(v^\mathsf{p},v^\mathsf{s}_1))\lor\exists v^\mathsf{p}(\lnot R(v^\mathsf{p},v^\mathsf{s}_0)\land\lnot R(v^\mathsf{p},v^\mathsf{s}_1))\bigr).
\end{multline*}
\item Compactness is not preserved under dense Chu transforms and, therefore, is not expressible in general by an inconsistency-flow formula. To see this, let $Y$ be an infinite set and $p\in Y$. Equip $Y$ with the particular point topology (see the book \cite{MR0507446} by Lynn Steen and J. Arthur Seebach) defined as $\tau(Y)=\Set{X\subseteq Y}{X=\emptyset\text{ or }p\in X}$. Then, the inclusion function $\{p\}\to Y$ gives a dense Chu transform from a compact space to a non-compact space.
\end{enumerate}
\end{examples}

The following definition aims to capture exactly how much compactness is preserved by dense Chu transforms. The terminology comes from the ``absolutely closed'' topological spaces introduced by Alexandroff and Urysohn~\cite{MR1512213}.

\begin{definition}
Let $\kappa\le\lambda$ be infinite cardinals. A Chu space is \emph{$\langle\kappa,\lambda\rangle$-absolutely closed} if it satisfies the following $\mathcal{L}^\mathrm{II}_{{(\lambda^{<\kappa})}^+,\lambda^+}$-sentence:
\begin{multline}\label{eq:hclosed}
\forall\Seq{v^\mathsf{s}_\alpha}{\alpha<\lambda}
\biggl[\exists v^\mathsf{p}\bigwedge_{\alpha<\lambda}\lnot R(v^\mathsf{p},v^\mathsf{s}_\alpha)\\ \lor\bigvee_{Z\in{[\lambda]}^{<\kappa}}\forall v^\mathsf{s}\bigl[\lnot\Con(v^\mathsf{s})\lor\exists v^\mathsf{p}\bigl(R(v^\mathsf{p},v^\mathsf{s})\land\bigvee_{\alpha\in Z}R(v^\mathsf{p},v^\mathsf{s}_\alpha)\bigr)\bigr]\biggr].
\end{multline}
\end{definition}

As for compactness, we relate this property with topology and logic.

\begin{remark}
Let $\langle X,\tau(X)\rangle$ be a topological space; then the following are equivalent:
\begin{itemize}
\item the Chu space $\langle X,\in,\tau(X)\rangle$ is $\langle\kappa,\lambda\rangle$-absolutely closed,
\item every open cover of $X$ of cardinality $\le\lambda$ has a subset of cardinality $<\kappa$ whose union is dense in $X$.
\end{itemize}
Let $\langle M,\vDash,S\rangle$ be an abstract logic which is closed under negation; then the following are equivalent:
\begin{itemize}
\item the Chu space $\langle M,\vDash,S\rangle$ is $\langle\kappa,\lambda\rangle$-absolutely closed,
\item given a set of sentences $T$ of cardinality $\le\lambda$, if for each $\Sigma\in{[T]}^{<\kappa}$ there exists a consistent sentence $\sigma$ such that every model of $\sigma$ is a model of $\Sigma$, then $T$ has a model.
\end{itemize}
\end{remark}

Since the sentence in \eqref{eq:hclosed} is inconsistency flow, from Proposition~\ref{prop:dense} we obtain the following corollary.

\begin{corollary}
Dense Chu transforms preserve $\langle\kappa,\lambda\rangle$-absolute closure of Chu spaces.
\end{corollary}

We observe that compactness implies absolute closure. Moreover, in many cases which occur naturally for abstract logics, the two properties coincide.

\begin{proposition}
Let $\kappa\le\lambda$ be infinite cardinals.
\begin{itemize}
\item If a Chu space is $\langle\kappa,\lambda\rangle$-compact, then it is $\langle\kappa,\lambda\rangle$-absolutely closed.
\item If a Chu space is $\langle\kappa,\lambda\rangle$-absolutely closed and admits complements and ${<}\kappa$-intersections, then it is $\langle\kappa,\lambda\rangle$-compact.
\end{itemize}
\end{proposition}
\begin{proof}
Given $Z\in{[\lambda]}^{<\kappa}$, the following holds in any Chu space:
\[
\forall\Seq{v^\mathsf{s}_\alpha}{\alpha\in Z}\biggl[\forall v^\mathsf{p}\bigvee_{\alpha\in Z}R(v^\mathsf{p},v^\mathsf{s}_\alpha)\longrightarrow\forall v^\mathsf{s}\bigl[\lnot\Con(v^\mathsf{s})\lor\exists v^\mathsf{p}\bigl(R(v^\mathsf{p},v^\mathsf{s})\land\bigvee_{\alpha\in Z}R(v^\mathsf{p},v^\mathsf{s}_\alpha)\bigr)\bigr]\biggr].
\]
Furthermore, the following holds in any Chu space which admits complements and ${<}\kappa$-intersections:
\[
\forall\Seq{v^\mathsf{s}_\alpha}{\alpha\in Z}\biggl[\forall v^\mathsf{p}\bigvee_{\alpha\in Z}R(v^\mathsf{p},v^\mathsf{s}_\alpha)\longleftrightarrow\forall v^\mathsf{s}\bigl[\lnot\Con(v^\mathsf{s})\lor\exists v^\mathsf{p}\bigl(R(v^\mathsf{p},v^\mathsf{s})\land\bigvee_{\alpha\in Z}R(v^\mathsf{p},v^\mathsf{s}_\alpha)\bigr)\bigr]\biggr].
\]
Each of the two points then follows easily.
\end{proof}

In analogy with the preservation results of the previous section, we aim to show that Proposition~\ref{prop:dense} is optimal for $\mathcal{L}^\mathrm{II}_{\omega,\omega}$-formulas. To do so, we shall use recursively saturated model pairs as in van Benthem's proof of Theorem~\ref{the-vb}. Recursively saturated structures were first studied in John Schlipf's PhD thesis \cite{MR2625642} and later appeared in a paper by Barwise and Schlipf~\cite{MR0403952}. Let us first recall their definition.

\begin{definition}[{Barwise and Schlipf~\cite[Definition~1.1]{MR0403952}}]
Let $\mathcal{L}$ be a finite language. An $\mathcal{L}$-structure $\mathfrak{M}$ is \emph{recursively saturated} if for every finite $C\subseteq M$ and every recursive set $\Gamma(x)$ of $\mathcal{L}_{\omega,\omega}$-formulas with parameters from $C$, if $\Gamma(x)$ is finitely satisfiable in $\mathfrak{M}$, then $\Gamma(x)$ is satisfiable in $\mathfrak{M}$.
\end{definition}

The next theorem is an existence result for recursively saturated structures.

\begin{theorem}[{Barwise and Schlipf~\cite[1.4(i)]{MR0403952}}]\label{theorem:bs}
Each structure in a finite language has a recursively saturated elementary extension of the same cardinality.
\end{theorem}

Let $\mathcal{L}=\{R_1,\dots,R_n\}$ be a finite relational language. Given two $\mathcal{L}$-structures $\mathfrak{M}$ and $\mathfrak{N}$, the \emph{model pair} $(\mathfrak{M},\mathfrak{N})$ is defined as the structure
\[
(\mathfrak{M},\mathfrak{N})=\bigl\langle M\cup N,M,N,R_1^\mathfrak{M},\dots,R_n^\mathfrak{M},R_1^\mathfrak{N},\dots,R_n^\mathfrak{N}\bigr\rangle.
\]
Each $\mathcal{L}_{\omega,\omega}$-formula $\varphi(\bm{x})$ has relativizations $\varphi'(\bm{x})$ and $\varphi''(\bm{x})$, obtained from $\varphi$ in a recursive way, such that for all $\bm{m}\in M$
\[
\mathfrak{M}\models\varphi[\bm{m}]\iff(\mathfrak{M},\mathfrak{N})\models\varphi'[\bm{m}]
\]
and for all $\bm{n}\in N$
\[
\mathfrak{N}\models\varphi[\bm{n}]\iff(\mathfrak{M},\mathfrak{N})\models\varphi''[\bm{n}].
\]
This technique of gluing structures together dates back to Abraham Robinson's proof of his consistency theorem \cite[Theorem~2.6]{MR0078307}. In the next lemma, we apply this to Chu spaces coded as two-sorted relational structures.

\begin{lemma}\label{lemma:sat}
Let $\langle X,r,A\rangle$ and $\langle Y,s,B\rangle$ be countable Chu spaces such that the model pair $(\langle X,r,A\rangle,\langle Y,s,B\rangle)$ is recursively saturated. Given finite tuples $\bm{x}\in X$ and $\bm{y}\in Y$ of matching length, as well as finite tuples $\bm{a}\in A$ and $\bm{b}\in B$ of matching length, the following are equivalent:
\begin{enumerate}
\item\label{lemma1} for every first-order inconsistency-flow formula $\gamma(\bm{v}^\mathsf{p},\bm{v}^\mathsf{s})$, if $\langle X,r,A\rangle\models\gamma[\bm{x},\bm{a}]$ then $\langle Y,s,B\rangle\models\gamma[\bm{y},\bm{b}]$;
\item\label{lemma2} there exists a dense Chu transform $(f,g)$ from $\langle X,r,A\rangle$ to $\langle Y,s,B\rangle$ such that $f(\bm{x})=\bm{y}$ and $g(\bm{b})=\bm{a}$.
\end{enumerate}
\end{lemma}
\begin{proof}
$(\ref{lemma1}\implies\ref{lemma2})$ For simplicity of notation, let us assume that the given tuples have length one: we denote them as single elements $x\in X$, $y\in Y$, $a\in A$, and $b\in B$. Since $X$ and $B$ are assumed to be countable, we can enumerate
\[
X=\Set{x_n}{n<\omega}\quad\text{and}\quad B=\Set{b_n}{n<\omega}
\]
in such a way that $x_0=x$ and $b_0=b$.

We shall recursively construct $\Seq{y_n}{n<\omega}$ in $Y$ and $\Seq{a_n}{n<\omega}$ in $A$ such that for each $n<\omega$ the following condition is satisfied:
\begin{multline}\label{starn}
\text{for every first-order inconsistency-flow formula }\gamma(v_0^\mathsf{p},\dots,v_n^\mathsf{p},v_0^\mathsf{s},\dots,v_n^\mathsf{s}),\\
\langle X,r,A\rangle\models\gamma[x_0,\dots,x_n,a_0,\dots,a_n]\implies\langle Y,s,B\rangle\models\gamma[y_0,\dots,y_n,b_0,\dots,b_n].
\end{multline}
For the base case, let $y_0=y$ and $a_0=a$. Condition \eqref{starn} is satisfied because of our assumption \eqref{lemma1}.

At stage $n<\omega$, assume we have constructed $\langle y_0,\dots,y_n\rangle$ and $\langle a_0,\dots,a_n\rangle$ satisfying \eqref{starn}. To get $y_{n+1}\in Y$, consider the following recursive\footnote{The set of all first-order inconsistency-flow formulas is recursive. Notice that $\Gamma(v^\mathsf{p})$ contains no mention of the relation of satisfiability in $\mathfrak{M}$ or $\mathfrak{N}$.} set of formulas:
\begin{multline*}
\Gamma(v^\mathsf{p})=\bigl\{\gamma'(x_0,\dots,x_{n+1},a_0,\dots,a_n)\rightarrow\gamma''(y_0,\dots,y_n,v^\mathsf{p},b_0,\dots,b_n)\bigm\vert\\
\gamma(v_0^\mathsf{p},\dots,v_n^\mathsf{p},v^\mathsf{p},v_0^\mathsf{s},\dots,v_n^\mathsf{s})\text{ is }\mathcal{L}^\mathrm{II}_{\omega,\omega}\text{-inconsistency flow}\bigr\}.
\end{multline*}
We claim that $\Gamma(v^\mathsf{p})$ is finitely satisfiable in $(\langle X,r,A\rangle,\langle Y,s,B\rangle)$. Indeed, let $\gamma_1,\dots,\gamma_m$ for some $m<\omega$ be inconsistency-flow formulas, as in the definition of $\Gamma(v^\mathsf{p})$, and suppose that for each $1\le i\le m$
\[
\langle X,r,A\rangle\models\gamma_i[x_0,\dots,x_{n+1},a_0,\dots,a_n].
\]
We let $\gamma(v_0^\mathsf{p},\dots,v_n^\mathsf{p},v_0^\mathsf{s},\dots,v_n^\mathsf{s})$ be the formula $\exists v^\mathsf{p}(\gamma_1\land\dots\land\gamma_m)$. Since $\gamma$ is also inconsistency flow and $\langle X,r,A\rangle\models\gamma[x_0,\dots,x_n,a_0,\dots,a_n]$, the inductive hypothesis \eqref{starn} guarantees that $\langle Y,s,B\rangle\models\gamma[y_0,\dots,y_n,b_0,\dots,b_n]$. This means that there exists $y\in Y$ such that for each $1\le i\le m$
\[
\langle Y,s,B\rangle\models\gamma_i[y_0,\dots,y_n,y,b_0,\dots,b_n],
\]
thus showing that $\Gamma(v^\mathsf{p})$ is finitely satisfiable in $(\langle X,r,A\rangle,\langle Y,s,B\rangle)$. By recursive saturation, $\Gamma(v^\mathsf{p})$ is satisfiable in $(\langle X,r,A\rangle,\langle Y,s,B\rangle)$ by some element $y_{n+1}\in Y$. Then it is clear by construction that
\begin{multline}\label{starn1}
\text{for every first-order inconsistency-flow formula }\gamma(v_0^\mathsf{p},\dots,v_{n+1}^\mathsf{p},v_0^\mathsf{s},\dots,v_n^\mathsf{s}),\\
\langle X,r,A\rangle\models\gamma[x_0,\dots,x_{n+1},a_0,\dots,a_n]\implies\langle Y,s,B\rangle\models\gamma[y_0,\dots,y_{n+1},b_0,\dots,b_n].
\end{multline}

To get $a_{n+1}\in A$, we proceed analogously and consider the recursive set of formulas
\begin{multline*}
\Delta(v^\mathsf{s})=\bigl\{\gamma'(x_0,\dots,x_{n+1},a_0,\dots,a_n,v^\mathsf{s})\rightarrow\gamma''(y_0,\dots,y_{n+1},b_0,\dots,b_{n+1})\bigm\vert\\
\gamma(v_0^\mathsf{p},\dots,v_{n+1}^\mathsf{p},v_0^\mathsf{s},\dots,v_n^\mathsf{s},v^\mathsf{s})\text{ is }\mathcal{L}^\mathrm{II}_{\omega,\omega}\text{-inconsistency flow}\bigr\}.
\end{multline*}
To see that $\Delta(v^\mathsf{s})$ is finitely satisfiable in $(\langle X,r,A\rangle,\langle Y,s,B\rangle)$, let $\gamma_1,\dots,\gamma_m$ for some $m<\omega$ be inconsistency-flow formulas as above and suppose that for each $1\le i\le m$
\[
\langle Y,s,B\rangle\models\lnot\gamma_i[y_0,\dots,y_{n+1},b_0,\dots,b_{n+1}].
\]
This time, we let $\gamma(v_0^\mathsf{p},\dots,v_{n+1}^\mathsf{p},v_0^\mathsf{s},\dots,v_n^\mathsf{s})$ be the inconsistency-flow formula $\forall v^\mathsf{s}(\gamma_1\lor\dots\lor\gamma_m)$. Since $\langle Y,s,B\rangle\models\lnot\gamma[y_0,\dots,y_{n+1},b_0,\dots,b_n]$, property \eqref{starn1} implies that $\langle X,r,A\rangle\models\lnot\gamma[x_0,\dots,x_{n+1},a_0,\dots,a_n]$. This means that there exists $a\in A$ such that for each $1\le i\le m$
\[
\langle X,r,A\rangle\models\lnot\gamma_i[x_0,\dots,x_{n+1},a_0,\dots,a_n,a],
\]
thus showing that $\Delta(v^\mathsf{s})$ is finitely satisfiable in $(\langle X,r,A\rangle,\langle Y,s,B\rangle)$. By recursive saturation, we can pick $a_{n+1}\in A$ which satisfies $\Delta(v^\mathsf{s})$ in $(\langle X,r,A\rangle,\langle Y,s,B\rangle)$. This completes the recursive construction of $\Seq{y_n}{n<\omega}$ and $\Seq{a_n}{n<\omega}$.

Now define $f\colon X\to Y$ and $g\colon B\to A$ by letting $f(x_n)=y_n$ and $g(b_n)=a_n$ for each $n<\omega$. Note that, if $x_m=x_n$ for some $m,n<\omega$, then $y_m=y_n$ by \eqref{starn} and the fact that $(v_0^\mathsf{p}=v_1^\mathsf{p})$ is an inconsistency-flow formula. Similarly, if $b_m=b_n$ for some $m,n<\omega$, then $a_m=a_n$ by \eqref{starn} again and the fact that $\lnot(v^\mathsf{s}_0=v^\mathsf{s}_1)$ is an inconsistency-flow formula. Hence, the functions $f$ and $g$ are well defined and, by the base step of the recursive construction, they satisfy $f(x)=y$ and $g(b)=a$.

Finally, we use \eqref{starn} to verify that $(f,g)$ is a dense Chu transform from $\langle X,r,A\rangle$ to $\langle Y,s,B\rangle$. First, to check the adjointness condition of Definition~\ref{def:chut}, take $x\in X$ and $b\in B$ and choose $m,n<\omega$ such that $x_m=x$ and $b_n=b$. Since both $R(v^\mathsf{p},v^\mathsf{s})$ and $\lnot R(v^\mathsf{p},v^\mathsf{s})$ are inconsistency-flow formulas, condition \eqref{starn} implies that
\[
\begin{split}
\langle x,g(b)\rangle\in r&\iff\langle X,r,A\rangle\models R(x_m,a_n)\\
&\iff\langle Y,s,B\rangle\models R(y_m,b_n)\\
&\iff\langle f(x),b\rangle\in s.
\end{split}
\]
For the density condition, let $b\in B$ and suppose there exists $y\in Y$ such that $\langle y,b\rangle\in s$. Choose $n<\omega$ such that $b_n=b$, hence in particular $\langle Y,s,B\rangle\models\exists v^\mathsf{p}R(v^\mathsf{p},b_n)$. Since $\lnot\exists v^\mathsf{p}R(v^\mathsf{p},v^\mathsf{s})$ is an inconsistency-flow formula, from \eqref{starn} it follows that $\langle X,r,A\rangle\models\exists v^\mathsf{p}R(v^\mathsf{p},a_n)$. Choose $x\in X$ such that $\langle x,a_n\rangle\in r$; then $\langle f(x),b\rangle\in s$, as we had to show.

$(\ref{lemma2}\implies\ref{lemma1})$ Follows from Proposition~\ref{prop:dense}.
\end{proof}

H. Jerome Keisler~\cite[\S 5]{MR3727403} applied the method of recursively saturated model pairs to obtain classic preservation results. In a similar fashion, we can apply Lemma~\ref{lemma:sat} to obtain our main theorem.

\begin{theorem}\label{theorem:preservationdense}
Let $T$ be a set of $\mathcal{L}^\mathrm{II}_{\omega,\omega}$-sentences; let $\varphi(\bm{v}^\mathsf{p},\bm{v}^\mathsf{s})$ and $\psi(\bm{v}^\mathsf{p},\bm{v}^\mathsf{s})$ be $\mathcal{L}^\mathrm{II}_{\omega,\omega}$-formulas. Then the following are equivalent:
\begin{enumerate}
\item\label{eq:pres1} If $\langle X,r,A\rangle$ and $\langle Y,s,B\rangle$ are models of $T$ and $(f,g)$ is a dense Chu transform from $\langle X,r,A\rangle$ to $\langle Y,s,B\rangle$, then for all $\bm{x}\in X$ and $\bm{b}\in B$
\[
\langle X,r,A\rangle\models\varphi[\bm{x},g(\bm{b})]\implies\langle Y,s,B\rangle\models\psi[f(\bm{x}),\bm{b}].
\]
\item\label{eq:pres2} There exists a first-order inconsistency-flow formula $\theta(\bm{v}^\mathsf{p},\bm{v}^\mathsf{s})$ such that
\[
T\vdash\forall\bm{v}^\mathsf{p}\forall\bm{v}^\mathsf{s}\bigl((\varphi\rightarrow\theta)\land(\theta\rightarrow\psi)\bigr).
\]
\end{enumerate}
\end{theorem}
\begin{proof}
$(\ref{eq:pres1}\implies\ref{eq:pres2})$ Let $\Phi$ be the set of inconsistency-flow consequences of $\varphi$, namely
\[
\Phi=\Set*{\alpha(\bm{v}^\mathsf{p},\bm{v}^\mathsf{s})}{\alpha\text{ is }\mathcal{L}^\mathrm{II}_{\omega,\omega}\text{-inconsistency flow and }T\vdash\forall\bm{v}^\mathsf{p}\forall\bm{v}^\mathsf{s}(\varphi\rightarrow\alpha)}.
\]
We intend to prove that taking $\theta=\bigwedge F$ for some finite $F\subseteq \Phi$ will satisfy the theorem. We first show that
\begin{equation}\label{eq:cons}
T\vdash\forall\bm{v}^\mathsf{p}\forall\bm{v}^\mathsf{s}\Bigl(\bigwedge\Phi\rightarrow\psi\Bigr).
\end{equation}
If $T$ is inconsistent, then \eqref{eq:cons} holds trivially. Otherwise, let $\langle Y,s,B\rangle$ be a model of $T$ and let $\bm{y}\in Y$ and $\bm{b}\in B$ be such that $\langle Y,s,B\rangle\models\alpha[\bm{y},\bm{b}]$ for each $\alpha\in\Phi$; we need to prove that $\langle Y,s,B\rangle\models\psi[\bm{y},\bm{b}]$. By the downward L{\"o}wenheim-Skolem theorem, we can assume that $\langle Y,s,B\rangle$ is countable.

Let us define
\[
\Sigma=\Set*{\varphi\land \lnot\gamma}{\gamma(\bm{v}^\mathsf{p},\bm{v}^\mathsf{s})\text{ is }\mathcal{L}^\mathrm{II}_{\omega,\omega}\text{-inconsistency flow and }\langle Y,s,B\rangle\models\lnot\gamma[\bm{y},\bm{b}]}.
\]
Now notice that the set of formulas $T\cup\Sigma$ is consistent. Indeed, whenever we take finitely many 
$\gamma_1,\dots,\gamma_m$ with each $\varphi\land\lnot\gamma_i\in\Sigma$, the set
\[
T\cup\Set{\varphi(\bm{v}^\mathsf{p},\bm{v}^\mathsf{s})\land\lnot\gamma_i(\bm{v}^\mathsf{p},\bm{v}^\mathsf{s})}{1\le i\le m}
\]
is consistent, for otherwise we would have $T\vdash\forall\bm{v}^\mathsf{p}\forall\bm{v}^\mathsf{s}\bigl(\varphi\rightarrow(\gamma_1\lor\dots\lor\gamma_m)\bigr)$ and therefore $(\gamma_1\lor\dots\lor\gamma_m)\in\Phi$, a contradiction. Since each of these sets is consistent, by compactness of two-sorted first-order logic the whole set is consistent.

Hence, there exists a countable model $\langle X,r,A\rangle$ of $T$ and $\bm{x}\in X$ and $\bm{a}\in A$ such that each of the following holds:
\begin{itemize}
\item $\langle X,r,A\rangle\models\varphi[\bm{x},\bm{a}]$;
\item for every first-order inconsistency-flow formula $\gamma(\bm{v}^\mathsf{p},\bm{v}^\mathsf{s})$, if $\langle X,r,A\rangle\models\gamma[\bm{x},\bm{a}]$ then $\langle Y,s,B\rangle\models\gamma[\bm{y},\bm{b}]$.
\end{itemize}
Now Theorem~\ref{theorem:bs} applied to the model pair $(\langle X,r,A\rangle,\langle Y,s,B\rangle)$ gives two countable elementary extensions 
\[
\langle X,r,A\rangle\preceq\langle X',r',A'\rangle\text{ and }\langle Y,s,B\rangle\preceq\langle Y',s',B'\rangle
\]
such that $(\langle X',r',A'\rangle,\langle Y',s',B'\rangle)$ is recursively saturated as a model pair. By Lemma~\ref{lemma:sat} applied to $\langle X',r',A'\rangle$ and $\langle Y',s',B'\rangle$, there exists a dense Chu transform $(f,g)$ from $\langle X',r',A'\rangle$ to $\langle Y',s',B'\rangle$ such that $f(\bm{x})=\bm{y}$ and $g(\bm{b})=\bm{a}$. From this, we conclude that $\langle Y,s,B\rangle\models\psi[\bm{y},\bm{b}]$, as we wanted to show.

This completes the proof of \eqref{eq:cons}. By compactness again, there exists a finite subset $F\subseteq\Phi$ such that $T\vdash\forall\bm{v}^\mathsf{p}\forall\bm{v}^\mathsf{s}\Bigl(\bigwedge F\rightarrow\psi\Bigr)$. Then it is sufficient to take $\theta$ to be $\bigwedge F$.

$(\ref{eq:pres2}\implies\ref{eq:pres1})$ Follows from Proposition~\ref{prop:dense}.
\end{proof}

In particular, setting $\varphi$ equal to $\psi$ in the above theorem and taking $T$ to be the empty set, we obtain that an $\mathcal{L}^\mathrm{II}_{\omega,\omega}$-formula is preserved under dense Chu transforms if and only if it is logically equivalent to a first-order inconsistency-flow formula.

\section{Ultrafilters as Chu spaces}\label{sec:ultrafilters}

We shall be coding ultrafilters as partially ordered sets, so we start with a known characterization of Chu transforms in this context. First of all, a partially ordered set $\bigl\langle P,\le^P\bigr\rangle$ can be represented as a separated and extensional Chu space $\bigl\langle P,{\le^P},P\bigr\rangle$. Chu transforms between such Chu spaces were first considered by Benado~\cite{MR0028817} and then systematically studied by J{\"u}rgen Schmidt~\cite{MR0057946}, who found the following characterization.

\begin{proposition}[{Schmidt~\cite[\S 8]{MR0057946}}]\label{proposition:schmidt}
Let $\bigl\langle P,\le^P\bigr\rangle$ and $\bigl\langle Q,\le^Q\bigr\rangle$ be partially ordered sets. For a pair of functions $f\colon P\to Q$ and $g\colon Q\to P$, the following two conditions are equivalent:
\begin{itemize}
\item $(f,g)$ is a Chu transform from $\bigl\langle P,\le^P,P\bigr\rangle$ to $\bigl\langle Q,\le^Q,Q\bigr\rangle$.
\item $f$ and $g$ satisfy each of the following conditions:
\begin{enumerate}
\item\label{ref:op} for all $p_0,p_1\in P$, if $p_0\le^P p_1$ then $f(p_0)\le^Q f(p_1)$;
\item for all $q_0,q_1\in Q$, if $q_0\le^Q q_1$ then $g(q_0)\le^P g(q_1)$;
\item for all $p\in P$, $p\le^P g(f(p))$;
\item\label{ref:dense} for all $q\in Q$, $f(g(q))\le^Q q$.
\end{enumerate}
\end{itemize}
\end{proposition}

Schmidt's characterization implies, in particular, that Chu transforms between partially ordered sets always satisfy the density condition.

\begin{corollary}
Let $\bigl\langle P,\le^P\bigr\rangle$ and $\bigl\langle Q,\le^Q\bigr\rangle$ be partially ordered sets. If $(f,g)$ is a Chu transform from $\bigl\langle P,\le^P,P\bigr\rangle$ to $\bigl\langle Q,\le^Q,Q\bigr\rangle$, then $(f,g)$ is dense.
\end{corollary}
\begin{proof}
Given $q\in Q$, we have to find $p\in P$ such that $f(p)\le^Q q$. Condition \eqref{ref:dense} of Proposition~\ref{proposition:schmidt} guarantees that taking $p=g(q)$ satisfies the requirement.
\end{proof}

An ultrafilter $U$ over a set $I$ is considered as a partially ordered set $\langle U,\supseteq\rangle$. We prove that, in this framework, Chu transforms correspond to an ordering of ultrafilters introduced in the sixties independently by Mary Ellen Rudin and Keisler. Let us recall its definition:

\begin{definition}
Suppose that $U$ and $V$ are ultrafilters over the sets $I$ and $J$, respectively. Then we say that $U$ is \emph{Rudin-Keisler reducible} to $V$, in symbols $U\le_\mathrm{RK}V$, if there exists a function $h\colon J\to I$ such that for every $X\subseteq I$,
\[
X\in U \iff h^{-1}[X]\in V.
\]
\end{definition}

Of course, it is sufficient for the function $h$ to be defined on some $Y\subseteq J$ such that $Y\in V$.

\begin{theorem}\label{th:RudinKeislerorder}
Let $U$ be an ultrafilter over a set $I$ and $V$ be an ultrafilter over a set $J$. Then the following are equivalent:
\begin{enumerate}
\item\label{rk1} there exists a Chu transform from $\langle U,\supseteq,U\rangle$ to $\langle V,\supseteq,V\rangle$;
\item\label{rk2} $U\le_\mathrm{RK}V$.
\end{enumerate}
\end{theorem}

\begin{proof}
$(\ref{rk1}\implies\ref{rk2})$ If $U$ is principal, then we are done, since any principal ultrafilter is Rudin-Keisler minimal. If $U$ is non-principal, then in particular the set $I$ must be infinite, hence we may assume without loss of generality that $I$ is an infinite cardinal $\kappa$.

Given a Chu transform $(f,g)$ from $\langle U,\supseteq,U\rangle$ to $\langle V,\supseteq,V\rangle$, we define
\[
Y=J\setminus\bigcap_{\alpha<\kappa}f(\kappa\setminus\{\alpha\}).
\]

\begin{claim}\label{cl:uno}
$Y\in V$.
\end{claim}
\begin{proof}\renewcommand{\qedsymbol}{\oldqedsymbol\textsubscript{Claim~\ref{cl:uno}}}
Suppose not. Then $\bigcap_{\alpha<\kappa}f(\kappa\setminus\{\alpha\})\in V$, hence $g\bigl(\bigcap_{\alpha<\kappa}f(\kappa\setminus\{\alpha\})\bigr)\in U$. In particular, this set is non-empty, so let us pick $\beta\in g\bigl(\bigcap_{\alpha<\kappa}f(\kappa\setminus\{\alpha\})\bigr)$. Then we have
\[
g\Biggl(\bigcap_{\alpha<\kappa}f(\kappa\setminus\{\alpha\})\Biggr)\not\subseteq\kappa\setminus\{\beta\}
\]
and therefore
\[
\bigcap_{\alpha<\kappa}f(\kappa\setminus\{\alpha\})\not\subseteq f(\kappa\setminus\{\beta\}),
\]
a contradiction.
\end{proof}

Let us define $h\colon Y\to\kappa$ as follows: for all $y\in Y$
\[
h(y)=\min\Set{\alpha<\kappa}{y\notin f(\kappa\setminus\{\alpha\})}.
\]
We claim that the function $h$ witnesses the Rudin-Keisler reducibility of $U$ to $V$.

\begin{claim}\label{cl:due}
If $X\in U$ then $f(X)\cap Y\subseteq h^{-1}[X]$, and therefore $h^{-1}[X] \in V$.
\end{claim}
\begin{proof}\renewcommand{\qedsymbol}{\oldqedsymbol\textsubscript{Claim~\ref{cl:due}}}
Once we show the first part of the claim, the second part will follow from Claim~\ref{cl:uno} and the fact that $V$ is an ultrafilter which contains $f(X)$ and $Y$.

For the first part, let $y\in f(X)\cap Y$, we want to show that $h(y)\in X$. Suppose not; then $X\subseteq\kappa\setminus\{h(y)\}$. By condition \eqref{ref:op} of Proposition~\ref{proposition:schmidt}, the function $f$ is inclusion-preserving, which gives $f(X)\subseteq f(\kappa\setminus\{h(y)\})$. But $y\in f(X)$ and therefore $y\in f(\kappa\setminus\{h(y)\})$, contradicting the definition of $h$.
\end{proof}

\begin{claim}\label{cl:tre}
For every $X\subseteq\kappa$, if $h^{-1}[X]\in V$ then $X\in U$.
\end{claim}
\begin{proof}\renewcommand{\qedsymbol}{\oldqedsymbol\textsubscript{Claim~\ref{cl:tre}}}
Suppose for a contradiction that $X\subseteq \kappa$ is such that $h^{-1}[X]\in V$ but $X\notin U$. Then $\kappa\setminus X\in U$
and so by Claim~\ref{cl:due}, $h^{-1}[\kappa\setminus X]\in V$. We obtain that $\emptyset=h^{-1}[X]\cap h^{-1}[\kappa\setminus X]\in V$, which is a contradiction.
\end{proof}

Now, Claim~\ref{cl:due} and Claim~\ref{cl:tre} together imply that $X\in U\iff h^{-1}[X]\in V$ for all $X\subseteq\kappa$, as we wanted to show.

$(\ref{rk2}\implies\ref{rk1})$ Let $h\colon J\to I$ be such that for all $X\subseteq I$
\[
X\in U\iff h^{-1}[X]\in V.
\]
It is straightforward to check that by letting
\[
\begin{split}
f\colon U &\longrightarrow V \\
X &\longmapsto h^{-1}[X]
\end{split}
\qquad\text{and}\qquad
\begin{split}
g\colon V &\longrightarrow U \\
Y &\longmapsto h[Y]
\end{split}
\]
we obtain a Chu transform $(f,g)$ from $\langle U,\supseteq,U\rangle$ to $\langle V,\supseteq,V\rangle$.
\end{proof}

Theorem~\ref{th:RudinKeislerorder} generalizes the following well-known result, credited by Rudin to Hewitt Kenyon.

\begin{corollary}[{Rudin~\cite[Theorem~1.5]{MR0080902}}] Let $U$ and $V$ be non-principal ultrafilters over $\omega$. If the partially ordered sets $\langle U,\supseteq\rangle$ and $\langle V,\supseteq\rangle$ are isomorphic, then there exists a bijective function $\pi\colon\omega\to\omega$ such that $V=\Set{\pi[X]}{X\in U}$.
\end{corollary}
\begin{proof} Let $f\colon U\to V$ be an isomorphism of partially ordered sets. Proposition~\ref{proposition:schmidt} implies that $\bigl(f,f^{-1}\bigr)$ is a Chu transform from $\langle U,\supseteq,U\rangle$ to $\langle V,\supseteq,V\rangle$ and $\bigl(f^{-1},f\bigr)$ is a Chu transform from $\langle V,\supseteq,V\rangle$ to $\langle U,\supseteq,U\rangle$. By a double application of Theorem~\ref{th:RudinKeislerorder} we obtain that $U\le_\mathrm{RK}V$ and $V\le_\mathrm{RK}U$, which is equivalent (see, e.g., Rudin~\cite[III.A]{MR0273581}) to the existence of a bijective function $\pi\colon\omega\to\omega$ such that $V=\Set{\pi[X]}{X\in U}$.
\end{proof}

\section{Graphs as Chu spaces}\label{section:graphs}

In this section, we briefly consider two possible realizations of graphs and related structures as Chu spaces.

\subsection{Strictly continuous functions between graphs}
A \emph{graph} is a pair $\langle V,E\rangle$ such that $E\subseteq{[V]}^2$. Each graph $\langle V,E\rangle$ can be naturally represented as an extensional Chu space $\langle V,\in,E\cup\{\emptyset\}\rangle$, where the points are the vertices of the graph, the relation is the membership relation, and the states consist of all the edges of the graph plus the empty state. To describe Chu transforms in this context, we recall a definition from a recent paper by Stuart Margolis and John Rhodes~\cite{MR4308405}, where strictly continuous functions on graphs are introduced.

\begin{definition}\label{def:strictcont}
Let $\langle V,E\rangle$ and $\langle V',E'\rangle$ be graphs and let $W\subseteq V$. A function $f\colon W\to V'$ is \emph{strictly continuous} if for every $e'\in E'$, either $f^{-1}[e']$ is empty or $f^{-1}[e']\in E$.
\end{definition}

We chose the terminology ``strictly continuous'' instead of the original ``strict continuous''. The interest in such functions comes from the following Lemma~\ref{th:MR}. Recall that a graph is \emph{connected} if any two distinct vertices are joined by a finite path.

\begin{lemma}[{Margolis and Rhodes~\cite[Lemma~2.9]{MR4308405}}]\label{th:MR}
Let $\langle V,E\rangle$ be a finite connected graph, let $W\subseteq V$, and let $f\colon W\to V$ be a strictly continuous function. If there exists $v\in V$ such that $\abs*{f^{-1}[\{v\}]}=1$, then $f$ is an automorphism of $\langle V,E\rangle$.
\end{lemma}

In analogy with Proposition~\ref{proposition:vickers}, we characterize Chu transforms between graphs as strictly continuous functions.

\begin{proposition}\label{proposition:graphchu}
Let $\langle V,E\rangle$ and $\langle V',E'\rangle$ be graphs. For a pair of functions $f\colon V\to V'$ and $g\colon E'\cup\{\emptyset\}\to E\cup\{\emptyset\}$, the following conditions are equivalent:
\begin{itemize}
\item $(f,g)$ is a Chu transform from $\langle V,\in,E\cup\{\emptyset\}\rangle$ to $\langle V',\in,E'\cup\{\emptyset\}\rangle$;
\item $f$ is a strictly continuous function and for every $e'\in E'\cup\{\emptyset\}$, $f^{-1}[e']=g(e')$.
\end{itemize}
\end{proposition}
\begin{proof}
If $(f,g)$ is a Chu transform and $e'\in E'\cup\{\emptyset\}$, then for every $v\in V$ we have
\begin{equation}\label{eq:gr}
v\in f^{-1}[e']\iff f(v)\in e'\iff v\in g(e')
\end{equation}
and therefore $f^{-1}[e']=g(e')$. In particular, $f^{-1}[e']\in E\cup\{\emptyset\}$, which means that $f$ is strictly continuous.

Conversely, if $f^{-1}[e']=g(e')$ for each $e'\in E'\cup\{\emptyset\}$, then again the two equivalences of \eqref{eq:gr} show that $(f,g)$ is a Chu transform.
\end{proof}

\begin{remark}\label{remark:graphchu}
For Chu transforms as in Proposition~\ref{proposition:graphchu}, the density condition translates to the requirement that $f^{-1}[e']\in E$ whenever $e'\in E'$. In other words, dense Chu transforms naturally correspond to strictly continuous functions for which the inverse image of each edge is non-empty.
\end{remark}

From Proposition~\ref{proposition:graphchu} and Remark~\ref{remark:graphchu}, we see that both Chu transforms and dense Chu transforms in this coding give rise to purely graph-theoretic notions.

\subsection{Adjointable maps between orthosets with $0$}
In recent work, Jan Paseka and Thomas Vetterlein~\cite{MR4912468} introduced orthosets with $0$ as a framework to investigate the concept of orthogonality in Hilbert spaces.

\begin{definition}[{Paseka and Vetterlein~\cite[Definition 2.1]{MR4912468}}]
An \emph{orthoset with $0$} is a triple $\langle X,\perp,0\rangle$ where $X$ is a non-empty set, $\perp$ is a binary relation on $X$, and $0\in X$, such that the following conditions are satisfied:
\begin{itemize}
\item for all $x_0,x_1\in X$, if $x_0\perp x_1$ then $x_1\perp x_0$;
\item for all $x\in X$, if $x\perp x$ then $x=0$;
\item for all $x\in X$, $0\perp x$.
\end{itemize}
\end{definition}

\begin{remark}
It is clear from the definition that if $\langle X,\perp,0\rangle$ is an orthoset with $0$ then $X\setminus\{0\}$ inherits the structure of a graph, in which two vertices $x_0$ and $x_1$ are joined by an edge if and only if $x_0\perp x_1$.
\end{remark}

We represent orthosets with $0$ as Chu spaces $\langle X,\perp,X\rangle$. In this realization, many properties of an orthoset with $0$ can be recast as properties of the associated Chu space. For example, $\langle X,\perp,0\rangle$ is \emph{irredundant} in the sense of \cite[Definition 2.6.(i)]{MR4912468} if and only if $\langle X,\perp,X\rangle$ is separated if and only if $\langle X,\perp,X\rangle$ is extensional.

\begin{definition}[{Paseka and Vetterlein~\cite[Definition 3.1]{MR4912468}}]
Let $\langle X,\perp,0\rangle$ and $\langle Y,{\perp},0\rangle$ be orthosets with $0$. Given $f\colon X\to Y$ and $g\colon Y\to X$, we say that $g$ is an \emph{adjoint} of $f$ if for every $x\in X$ and every $y\in Y$
\[
x\perp g(y)\iff f(x)\perp y.
\]
\end{definition}

The following characterization is immediate from the definitions.

\begin{proposition}\label{proposition:ortho}
Let $\langle X,\perp,0\rangle$ and $\langle Y,\perp,0\rangle$ be orthosets with $0$. For a pair of functions $f\colon X\to Y$ and $g\colon Y\to X$, the following conditions are equivalent:
\begin{itemize}
\item $(f,g)$ is a Chu transform from $\langle X,\perp,X\rangle$ to $\langle Y,\perp,Y\rangle$;
\item $g$ is an adjoint of $f$.
\end{itemize}
\end{proposition}

We note that, unlike the cases considered previously, the two conditions of Proposition~\ref{proposition:ortho} are symmetric in $f$ and $g$, meaning that $(f,g)$ is a Chu transform from $\langle X,\perp,X\rangle$ to $\langle Y,\perp,Y\rangle$ if and only if $(g,f)$ is a Chu transform from $\langle Y,\perp,Y\rangle$ to $\langle X,\perp,X\rangle$.

\section{Conclusion and further directions}

We have investigated the concepts of Chu spaces and Chu transforms, coming from category theory, for which we showed the use in various parts of logic and mathematics in general. These include abstract logics, topology (in particular regarding compactness properties, see \S\ref{sec:inconflow}), the theory of ultrafilters (\S\ref{sec:ultrafilters}), and graph theory (\S\ref{section:graphs}). We have shown that many classically considered notions in mathematics, such as the Rudin-Keisler order between ultrafilters, are special cases of Chu transforms.

As a main contribution of the paper, in \S\ref{sec:Benthem} and \S\ref{sec:inconflow} we have established a number of preservation results for classes of $\mathcal{L}^\mathrm{II}_{\infty,\infty}$-formulas under various kinds of Chu transforms. Our main theorem in this respect is the characterization of $\mathcal{L}^\mathrm{II}_{\omega,\omega}$-formulas preserved under dense Chu transforms, Theorem~\ref{theorem:preservationdense}.

Feferman observed in \cite[Remark 2]{feferman:johan99} that his proof of van Benthem's Theorem~\ref{the-vb} works also for countable admissible fragments of $\mathcal{L}_{\omega_1,\omega}$. We do not know to what extent Theorem~\ref{theorem:preservationdense} extends to larger infinitary languages.

\begin{question}
If an $\mathcal{L}^\mathrm{II}_{\infty,\omega}$-formula is preserved under dense Chu transforms, must it be logically equivalent to an inconsistency-flow formula?
\end{question}

To address this question, it is natural to look at the techniques of Barwise and van Benthem~\cite{MR1777793}, which yield a generalized interpolation theorem for $\mathcal{L}_{\infty,\omega}$.

Another interesting direction is to extend our analysis from $2$-valued Chu spaces to $k$-valued Chu spaces, namely triples $\langle X,r,A\rangle$ with $r\colon X\times A\to k$. The interest in such spaces is motivated by Vaughan Pratt~\cite[Theorem~14]{MR1265061}, who constructed a full concrete embedding from the category of $n$-ary relational structures and homomorphisms into the category of $2^n$-valued Chu spaces and Chu transforms. For this reason, $2^n$-valued Chu spaces seem well suited to express the notion of homomorphism between relational structures and the corresponding preservation results, such as Rossman's homomorphism preservation theorem \cite[Theorem~5.16]{MR2444913}. More recently, Shengyang Zhong~\cite{zhong} used $3$-valued Chu spaces to formalize the intuition behind quantum logics and Holliday's fundamental logic \cite{holliday}.

\begin{question}
Which properties are preserved by Chu transforms between $k$-valued Chu spaces?
\end{question}

In future work, we also hope to give syntactic characterizations and preservation theorems for two other related category-theoretic concepts: Dialectica morphisms and generalized Galois-Tukey connections.

\section*{Acknowledgements}
We thank Tom Benhamou for pointing out the connection between Rudin~\cite[Theorem~1.5]{MR0080902} and our Theorem~\ref{th:RudinKeislerorder}. We also thank the anonymous referee, whose report was particularly helpful. It is thanks to a suggestion in the report that we considered recursively saturated model pairs and were able to significantly shorten the proof of Theorem~\ref{theorem:preservationdense}, which originally was proved using the method of elementary chains.

\end{document}